\documentclass{amsart}
\usepackage{amsmath,amscd, amssymb}
\usepackage{amsmath}
\usepackage{caption}
\usepackage{color}
\usepackage{enumitem}
\usepackage{float}
\usepackage[bookmarksnumbered=true]{hyperref} 
\hypersetup{colorlinks = true, citecolor = blue}
\usepackage{natbib}
\usepackage{stmaryrd}
\usepackage{subcaption}
\usepackage{tikz}
\usepackage{xcolor,colortbl}
\usepackage{xcolor}
\usepackage[utf8]{inputenc}
\usepackage{color,soul}

\newtheorem{thm}{Theorem}

\newtheorem{lem}[thm]{Lemma}
\newtheorem{prop}[thm]{Proposition}
\numberwithin{equation}{section}
\newcommand\skiplines[1]{\vspace{#1\baselineskip}}

\definecolor{Gray}{gray}{0.85}
\definecolor{dgn}{rgb}{0.0, 0.5, 0.0}

\newcommand{\hh}{H$^{**}_2$}
\newcommand{\HHH}{H$^*_2$  }
\newcommand{\HHHH}{H$^*_2$}
\newcommand{\ie}{\textit{i.e.}}

\usepackage{lipsum}
\makeatletter
\g@addto@macro{\endabstract}{\@setabstract}
\newcommand{\authorfootnotes}{\renewcommand\thefootnote{\@fnsymbol\c@footnote}}%
\makeatother

\begin{document}
\title[Population collapse]
{Population collapse in Elite-dominated societies:\\
A differential equations model without differential equations}
\maketitle
\begin{center}
\authorfootnotes
  Naghmeh Akhavan\textsuperscript{1}, James A. Yorke\textsuperscript{2} \par \bigskip

  \textsuperscript{1}Department of Mathematics, \ University of Guilan, Iran \par
  \textsuperscript{2} University of Maryland, College Park, USA \par \bigskip
 
\end{center}
\footnotetext[1]{\textit{E-mail address:}{\ Naghmeh.Akhavan@gmail.com}}
\footnotetext[2]{\textit{E-mail address:}{\ Yorke@umd.edu}}

\begin{abstract}
The HANDY model of Motesharrei, Rivas, and Kalnay examines interactions 
with the environment by human populations, both between poor and rich people, \textit{\ie}, ``Commoners'' and ``Elites''. The Elites control the society's wealth and consume it at a higher rate than Commoners, whose work produces the wealth.
We say a model is “Elite-dominated” when the Elites' per capita population change rate is always at least as large as the Commoners'.
We can show the HANDY model always exhibits population crashes for all choices of parameter values for which it is Elite-dominated.
But any such model with explicit equations raises questions of how the resulting behaviors depend on the details of the models. How important are the particular design features codified in the differential equations? 
In this paper, we first replace the explicit equations of HANDY with differential equations that are only described conceptually or qualitatively --- using only conditions that can be verified for explicit systems. 
Next, we discard the equations entirely, replacing them with qualitative 
conditions, and we prove these conditions imply population collapse must occur. In particular, one condition is that the model is Elite-dominated.
We show that the HANDY model with Elite-dominated parameters satisfies our hypotheses and thus must undergo population collapse.
Our approach of introducing qualitative mathematical hypotheses can better show the underlying features of the model that lead to collapse. We also ask how societies can avoid collapse.
\end{abstract}
\maketitle

\section{Introduction}

Throughout history and in prehistory, civilizations have risen and then collapsed.
There is a large body of literature investigating societal collapse (\citep{turchin2009secular,shennan2013regional,goldberg2016post,motesharrei2016modeling,turchin2018historical}, and references therein). \citet{diamond2005collapse} attributes collapse to four main causes: 
environmental damage, climate change, hostile neighbors, and trade partners,
some of which he reports were exacerbated by an Elite--Commoner stratification as in Greenland and Easter Island. 
Diamond investigates these and other factors for a variety of societies, from the Mayan people and isolated island populations to regions of ancient Egypt, India, and China, 
for which there are historical records of populations collapse \citep{chu1994famine,stark2006funan}. 
Countries with apparently similar circumstances can have different outcomes, some surviving longer than the others. For example, while Easter Island experienced a population collapse, the 
Pacific Island of Tikopia with a similar environment had a drastically different outcome.
Tikopia maintained an average population change rate of zero and a sustainable rate of resource use \citep{erickson2000resource}. 

Motesharrei, Rivas, and Kalnay argued in \cite{motesharrei2014human} 
that an Elite--Commoner economic stratification can sometimes by itself lead to collapse, a view that our paper supports and focuses on. 

The \cite{alfred1925lotka} and \cite{volterra1927variazioni} models of predator and prey (wolves--rabbits) can exhibit sustained periodic oscillations ( see also \cite{smith1992economic}). \cite{brander1998simple} created a Lotka--Volterra model with humans as the predator and regenerating resources as the prey, with feast and famine oscillations. 
The HANDY (Human And Nature DYnamics) model, \cite{motesharrei2014human} is a 4-dimensional differential equations model in which there are two human populations, ``Elites'' and ``Commoners", whose population sizes are $E(t)$ and $C(t)$, and two prey elements, regenerating resources and wealth. 
We will often refer to regenerating resources as ``food'' and wealth as ``stored food''
though these categories could include trees and other types of biomass.
Commoners can be thought of as the workers who harvest or hunt or gather all the food for the community while the Elites do not work but control the distribution of stored food. 
In this paper, we investigate and generalize HANDY to better understand societal interactions that can cause collapse.

\textbf{ ``Elite-dominated" models.}
We restrict attention in this paper to what we define  as  {\bf Elite-dominated} models, those for which (1) increased consumption of food never decreases per capita population change rate, and (2) Elite individuals always consume more food than Commoners.
Our models might not apply to societies where people are able to plan and manage the growth of their populations and their exploitation of their resources.

Sec.~\ref{sec:HANDY} presents the HANDY model in detail, with some modifications.
Note that our notation is different from \citet{motesharrei2014human}.
All of our results for HANDY are for Elite-dominated choices of parameters. 

Figure~\ref{fig:dynamics} displays some behaviors of the HANDY model. The left side has 
$C(0) > 0$ with $E(t) \equiv 0$ and shows an undamped oscillation. The right side shows the result when $E(0) > 0$ is a small positive number.  Initially both show similar oscillations, but when $E(0) > 0$, $E(t)/C(t)$ eventually increases to the point where the Commoners cannot access  enough food to sustain themselves. Then $C$ and $E$ decrease toward $0$ despite a slowly growing food resource that could be hunted or gathered. When $E(0)>0$, we prove our Elite-dominated models always exhibit population collapse  (Sec.~\ref{sec:HANDY}). That raises the critical question of how societies that avoid collapse are organized (see Sec.~\ref{no-Collapse}).
 
\textbf{ Overview.}
By \textbf{population collapse} we mean that the two human populations die out,
\begin{align*}
    C(t) \mbox{ and } E(t) \to 0 \mbox{ as } t\to\infty  \mbox{ whenever }  C(0)>0 \mbox{ and } E(0)>0.
\end{align*}
This is an extreme type of collapse. Biologists and archaeologists understandably have a much more relaxed definition. They declare that a collapse has occurred whenever the population(s) become quite small. 

While the variables that we call $C(t)$ and $E(t)$ are meant to represent population sizes of Commoners and Elites, such variables 
could represent the accumulated wealth of Elites and Commoners in a modified model,
hence directly modelling economic stratification.

\citet{motesharrei2014human} demonstrated numerically several collapse scenarios for HANDY. Our first goal was to investigate whether there are choices of HANDY parameters for which we can prove a population collapse occurs. And we eventually established that to our satisfaction. 
Then we were faced with the difficult question of determining how the population collapse depends on the choice of values for the 9 parameters, and, more broadly, on the choice of 5 functions that appear in the differential equations, and even on the number of equations. \citet{motesharrei2014human} described that overdepletion of Nature, high levels of inequality, and overpopulation beyond carrying capacity can lead to collapses in HANDY. We think our approach elucidates the causes of the HANDY model's population collapse and thus can generalize the conditions that might lead to sustainability or collapse. 
\begin{figure}
\begin{subfigure}{.47\linewidth}
\centering
\includegraphics[width=1.\linewidth]{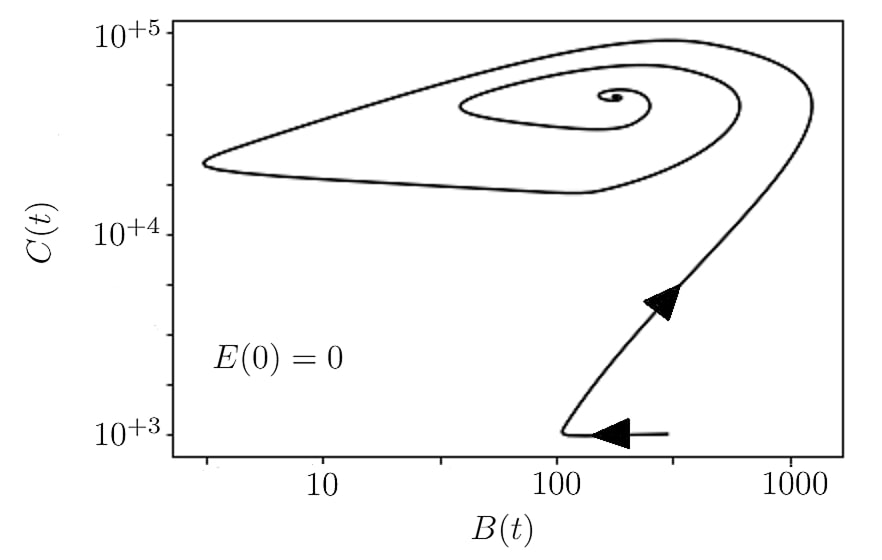}
\caption{}
\end{subfigure}%
\begin{subfigure}{.47\linewidth}
\centering
\includegraphics[width=1.\linewidth]{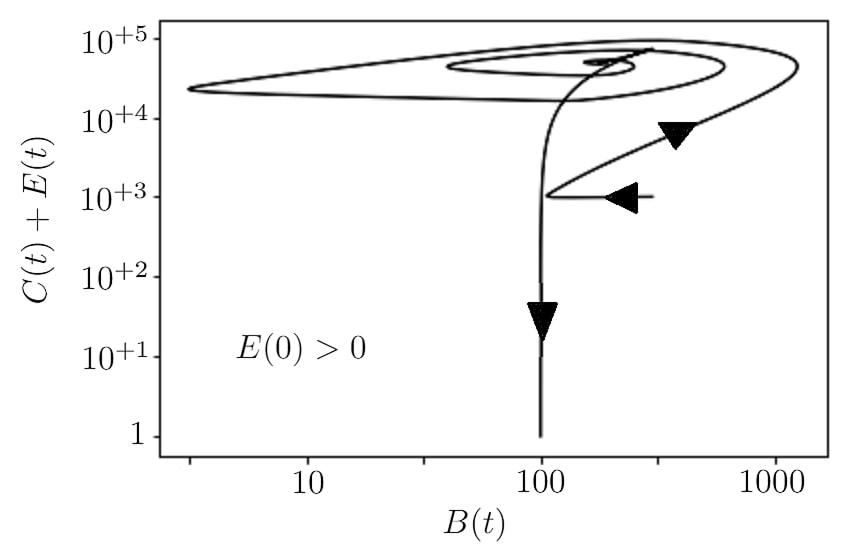}
\caption{}
\end{subfigure}\\[1ex]
\begin{subfigure}{.45\linewidth}
\centering
\includegraphics[width=1.\linewidth]{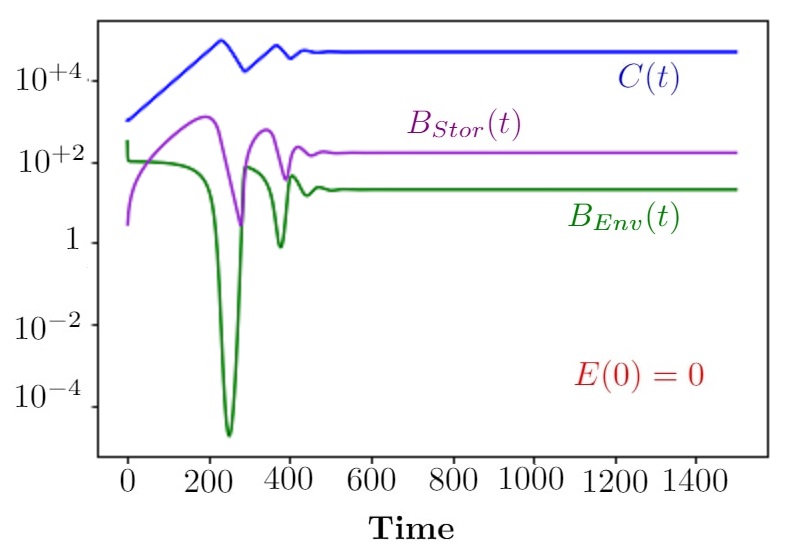}
\caption{}
\end{subfigure}%
\begin{subfigure}{.45\linewidth}
\centering
\includegraphics[width=1.\linewidth]{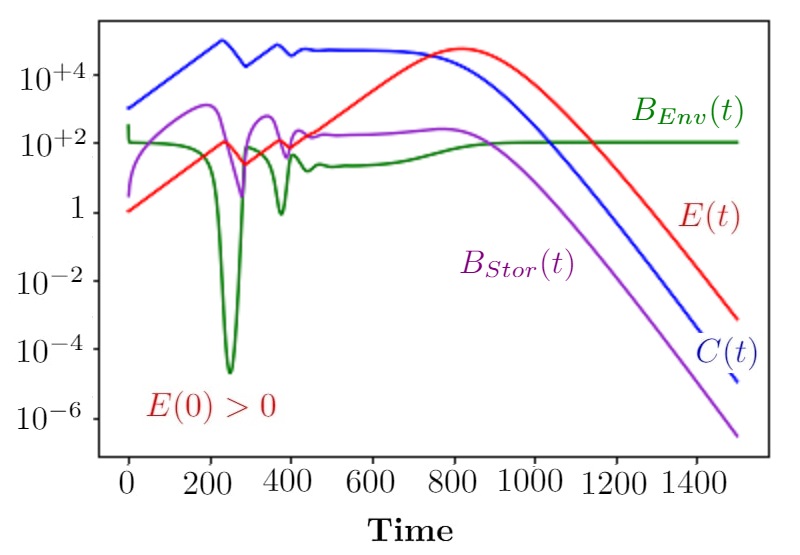}
\caption{}
\end{subfigure}\\[1ex]
\begin{subfigure}{.45\linewidth}
\centering
\includegraphics[width=1.\linewidth]{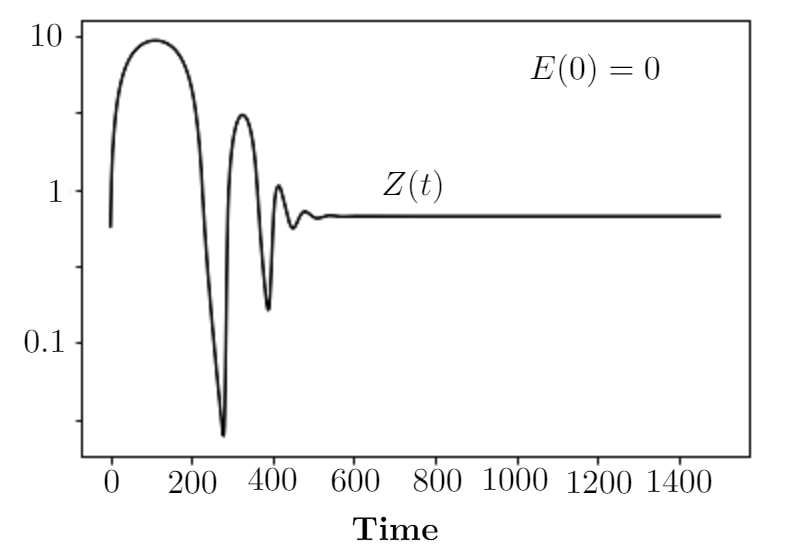}
\caption{}
\end{subfigure}%
\begin{subfigure}{.45\linewidth}
\centering
\includegraphics[width=1.\linewidth]{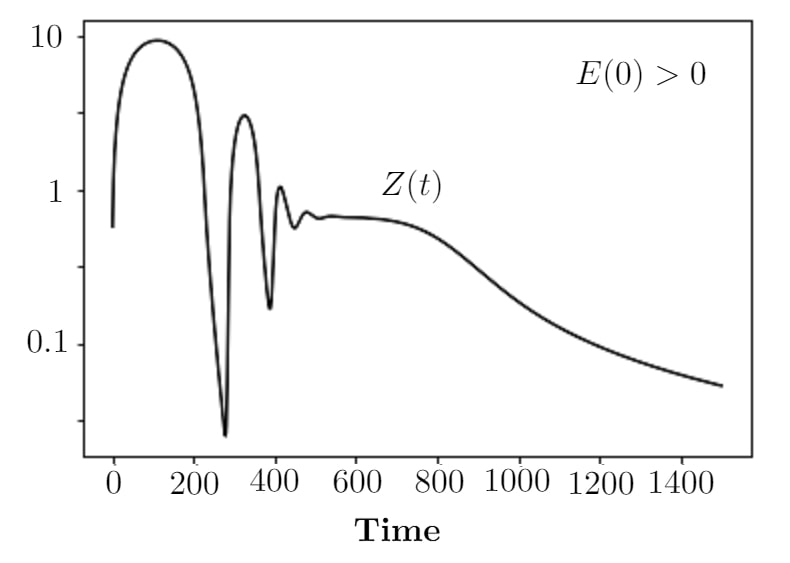}
\caption{}
\end{subfigure}
\caption{\textbf{ HANDY model asymptotic behavior.} 
The left panels have $E = 0$ and the right panels $E>0$. $B(t)$ represents stored food plus food in the fields.
\textbf{{(Left Panels)}} The population $C(t)$, and food resource $B(t)$ approach equilibrium.
\textbf{{(Right Panels)}} $E(0)>0$ with the initial ratio, $\frac{C(0)}{E(0)}=10^{+3}$. In Elite-dominated societies, \textit{\ie}, where Elites' population change rate is larger than or equal to that of Commoners',  the populations collapse.
\textbf{{(A,B)}} Total population is plotted vertically against $B(t)$.
\textbf{{(C,D)}} Populations evolving over time.
\textbf{{(E,F)}}  Per capita food supply $Z(t)$ \eqref{Z} for Commoners is plotted against time $t$. 
Initial conditions for all panels are in Table~\ref{table-params}. }
\label{fig:dynamics}
\end{figure}

\textbf{H: Verifiable qualitative hypotheses for differential equations.} In Sec.~\ref{sec:handy}, we substitute HANDY's four explicit equations with the following three generic population change equations (\ref{G-model} plus five qualitative attributes, hypotheses, or assumptions.   We use $'$ to denote time derivative, $\frac{d}{dt}$.
\begin{equation}
        \begin{cases}
        B'= B \cdot R_B(B, C, E), \mbox{   (total food resources)} \\
        C'= C \cdot R_C(B, C, E), \mbox{   (Commoners)}\\
        E'= E \cdot R_E(B, C , E), \mbox{   (Elites)},
        \end{cases}
        \label{G-model}
\end{equation}
We refer to these equations together with five assumptions H$_1$, H$_2$, H$_3$, H$_B$, and H$_Z$ in Sec.~\ref{sec:handy} as the \textbf{{H model}}, where \textbf{H} stands for Hypotheses.  There are no explicit formulas for $R_B$, $R_C$, and $R_E$. 
With a surprisingly difficult proof, we show that when $E(0)>0$ in an Elite-dominated society, the \textbf{{H model}} always exhibits population collapse; see Thm.~\ref{mainThm}.
Note that there are two key assumptions, H$_2$ and H$_3$, which imply 
\begin{align}
    E'(t)>0 \ \ \ \mbox{when} \ \ \ C'(t)=0,
\label{eq:h2}
\end{align}
and for all $t \ge 0$,
\begin{align}
    \frac{E'}{E}(t) \ge \frac{C'}{C}(t).
\label{eq:h3}
\end{align}

Assumption~\eqref{eq:h2} implies that the Elite population is still growing when Commoner population change becomes zero. Assumption~\eqref{eq:h3} implies that the per capita population change rate of Elites is always greater than or equal to that of Commoners.

Because the HANDY model has four equations instead of three, it is not a special case of the H model. 
But can we generalize Thm.~\ref{mainThm} and omit the equations in (\ref{G-model}) altogether, thereby include HANDY as a special case?

\textbf{ H$^*$: Verifiable Qualitative Hypotheses without differential equations.} 
In Sec.~\ref{sec:no_equations}
we eliminate the three equations (\ref{G-model}) and $R_B$, $R_C$, and $R_E$ by using more refined hypotheses H$^*_1$, \HHHH, H$_3$, H$_B$, and H$_Z$, which by themselves guarantee population collapse. We call this new version the ``\textbf{{H$^*$ qualitative model}}" (or \textbf{{H$^*$ model}}). 
The modified hypotheses are only about the functions $B(t)$, $C(t)$, and $E(t)$ for all $t \ge 0$.

We require that the all assumptions of H and H$^*$ must be directly verifiable for systems like HANDY. Hence it would be unacceptable to have an assumption such as \emph{the model's solutions exhibit population collapse}, unacceptable because it would be quite difficult to verify this condition.

The $B(t), C(t) \ \mbox{and} \ E(t)$ in H$^*$ can also represent composite components of more complex systems. 
The food supply, $B$, might include a wide variety of species of plants and animals, both hunted and harvested, which might be rather difficult to realistically model with explicit equations.
Furthermore, there could be periods of climate variations and other time-dependent fluctuations, many of which can be included under H$^*$.

We prove that the H model (Prop.~{\ref{RelationBCEeq&qu}}) and the Elite-dominated HANDY model (Prop.~{\ref{hyps-4-dim}}) are special cases of the H$^*$ model. In particular, both models must result in population collapse. 

\section{The H (equation) model}\label{sec:handy}
We introduce a model motivated by the HANDY model, \citep{motesharrei2014human}.
HANDY describes situations where there are two classes of people, called Commoners, $C(t)$, and Elites, $E(t)$.  We refer to the per capita population change rates as {\textbf{change rates}}. The Commoners do the work of growing and storing, or hunting and gathering food, $B(t)$. 
When food is plentiful, Elite-dominated HANDY assumes that the change rates, $\frac{C'}{C}$ and $\frac{E'}{E}$, are equal and positive. When food is scarce, the change rates are negative.
We began by aiming for a minimal collection of hypotheses. Our initial set was far more complex but as this project proceeded, the list simplified. 
Notice for example that we have no assumptions about how change rates depend on the food supply. Here we present a small set of verifiable hypotheses for which we can show that HANDY satisfies.  

We refer to $\frac{B'}{B}$, $\frac{C'}{C}$, and $\frac{E'}{E}$ as change rates of $B, C$ and $E$, omitting the implied ``per capita".
This system is defined on $\boldsymbol{\Omega :=\{X = (B, C, E) : B,C, E \geq 0\}}$.

We will use the following hypotheses. 
\skiplines{1}

\begin{itemize}
  \item[\textbf{ H$_1$.}] (What is $X(t)=(B,C,E)(t)$?)\\  \label{TrapHyp1}
{\it The functions $R_B, R_C,R_E:\Omega\to\mathbb R$ are continuously differentiable; the functions $B,C,E: [0, \infty) \to (0, \infty)$ are a solution of (\ref{G-model}). }
  \item[\textbf{ H$_2$.}]  (When there is a “mild food shortage”,\textit{i.e.}, when $C' = 0$, $E$ is increasing.)\\{\it $R_E>0$ when $R_C=0$.}
  \item[\textbf{ H$_3$.}] (Elites' population change rate is always at least as large as the Commoners'.)\\ {\it $\frac{E'}{E} \geq \frac{C'}{C}$.}
  \item[\textbf{ H$_B$.}] (Each initial point $X(0)$ is in a trapping region). \\
{\it Each trajectory is bounded.}
  \item[\textbf{ H$_Z$.}] 
(The Elite population is totally dependent on food gathered by Commoners.)\\
{\it If $C(t) \to 0$, then $E(t) \to 0$ as $t \to \infty$.} 
\end{itemize}

\bigskip

We say $X(t)=(B,C,E)(t)$ is a \textbf{{trajectory}} if $B,C,E: [0, \infty) \to (0, \infty)$ are continuously differentiable.
We refer to Eqs. (\ref{G-model}) under the hypotheses H$_1$, H$_2$, H$_3$, H$_B$, and H$_Z$ as an \textbf{ H model}. If a trajectory $(B, C, E)$ is a solution of an H model, we say it is an \textbf{{H trajectory}}.

An \textbf{{equilibrium}} for Eqs. (\ref{G-model}) is a state $B_e, C_e, E_e$ for which $B'=C'=E'=0$.

\begin{prop} \label{equilibriumProp}
There exists no equilibrium with $C>0$ and $E>0$ that satisfies H$_1$ and H$_2$.
\end{prop}
\begin{proof}
Suppose there is $X_e=(B_e, C_e, E_e)$ an equilibrium point with $C>0$ and $E>0$. 
Then, $R_C(X_e) = R_E(X_e) = 0$. But by Hyp.~H$_2$, $R_C=0$ implies $R_E>0$, a contradiction. Thus there is no such equilibrium.
\end{proof}

The above result is trivial, but the following result is far more difficult to prove.

\begin{thm} \label{mainThm} \textbf{{[Population collapse of H-trajectories]}}
Assume $(B,C,E)(t)$ is an H trajectory.
Then $C(t) \to 0$ and $E(t) \to 0$ as $t \to \infty$. 
\end{thm}

The proof is in Sec.~\ref{theorem2}.

\section{The H\texorpdfstring{$^*$} {TEXT} (equationless) model
\label{sec:no_equations}
}
We use the following hypotheses to generalize H$_1$ and H$_2$ so that no differential equations are needed.

\bigskip
\begin{itemize}
  \item[\textbf{{H$^*_1$.}}] (What is $(B,C,E)(t)$?)\\ {\it The functions $B, C, E: [0, \infty) \to (0, \infty)$ are continuously differentiable. {$X(t)=(B,C,E)(t)$ is a trajectory. If the trajectory is bounded, then $\sup_{t \ge 0}|X'(t)|< \infty$, and $\frac{C'}{C}$ and $\frac{E'}{E}$ are uniformly continuous.} } 
  \item[{{\textbf{H$_2^*$.}}}] {\it For a bounded trajectory, there exist $\varepsilon_2^*>0$ and $\delta_2^*>0$ such that $|\frac{C'}{C}| \le \delta_2^*$ implies $\frac{E'}{E} - \frac{C'}{C}>\varepsilon_2^*$.}
\end{itemize}

\bigskip
As with $\varepsilon_2^*$, which is related to H$_2$ and H$_2^*$, a subscript often suggests which hypothesis it is related to.

\bigskip 

\begin{figure} 
    \centering
    \includegraphics[width=0.75\textwidth]{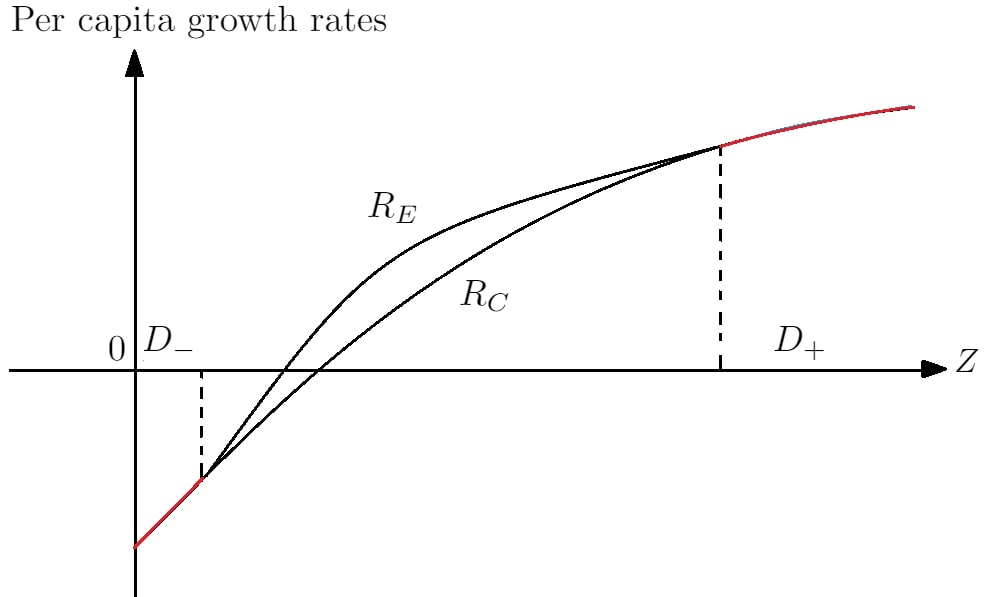}
    \caption{\textbf{{Per capita change rate for populations.}} These graphs are simplified to illustrate a possible choice of per capita change rates for $C$ and $E$ consistent with H$_2$ and H$_3$. Here $Z$ is the per capita food supply for Commoners. The graph represents the special case where $R_E$ and $R_C$ can be written as a function of $Z$. This figure illustrates the idea that $R_E$ can be equal to $R_C$ for a variety of situations but not when $R_C=0$. There can be two regions $D_-$ and $D_+$ (where the curves are red) where the change rates are equal. For $D_-$ the change rates are equal but negative. For $D_+$ change rates are equal but positive.}
   \label{Fig:rela}
\end{figure}

If a trajectory $(B, C, E)(t)$ satisfies H$^*_1$, H$_2^*$, H$_3$, H$_B$, and H$_Z$, we say it is an \textbf{{H$^*$ trajectory}}. 
Notice in particular that an \textbf{{H$^*$ trajectory}} is not assumed to satisfy any differential equations; it has no analog to Eqs.~\eqref{G-model}.

\begin{thm} \label{mainThmstar} \textbf{{[Population Collapse for H$^*$ qualitative trajectories]}}\\
Assume $(B,C, E)$ is an H$^*$ trajectory. 
Then $C(t) \to 0$ and $E(t) \to 0$ as $t \to \infty$. 
\end{thm}
The proof of this result is in Sec.~\ref{proofmainThmstar}.
\skiplines{1}

Proposition~\ref{RelationBCEeq&qu} converts Thm.~\ref{mainThm} into a Corollary of Thm.~\ref{mainThmstar}. Hence there is no need to prove Thm.~\ref{mainThm} separately. Nonetheless, we show in Sec.~\ref{theorem2} how the proof of collapse becomes simpler when we have an autonomous differential equation, that is \eqref{G-model}. 

\begin{prop} \label{RelationBCEeq&qu}
Each H trajectory is an H$^*$ trajectory. 
\end{prop}
\begin{proof}
Let $X(t)=(B, C, E)(t)$ be an H trajectory; hence it satisfies (\ref{G-model}) and H$_1$, H$_2$, H$_3$, H$_B$, and H$_Z$. 

\textbf{{(H $\Longrightarrow$ H$^{*}_1$).}} 
H$_B$ says there exists $\Omega_{\text cpt}$ a compact subset of $\Omega$ that contains the trajectory for all $t \ge 0$.
Then H$_1$ implies $R_C$ is uniformly continuous and $|R_C|=|\frac{C'}{C}|$ is bounded on $\Omega_{\text cpt}$ and $X(t)$ is uniformly continuous since $|X'(t)|$ is bounded on $\Omega_{\text cpt}$.
Since $\frac{C'}{C}(t)=R_C(X(t))$ is the composition of two uniformly continuous functions (\textit{i.e.} $R_C$ and $X(t)$), $\frac{C'}{C}(t)$ is uniformly continuous on $\Omega_{cpt}$, which is one requirement of ${\text H}^{*}_1$. 

\textbf{{(H $\Longrightarrow$ \HHH).}} Suppose \HHH  is false. \\
Then there exists a sequence $X_n = (B_n, C_n, E_n) \in \Omega_{\text cpt}$ such that $\frac{C'_{n}}{C_n} \to 0$ and 
\begin{align}
    R_E(X_n) - R_C(X_n) = \frac{E'_n}{E_n} - \frac{C'_n}{C_n} \to 0 \ \mbox{as} \ n \to \infty.
\end{align}
By compactness, $X_n$ has a limit point $X^0$. It follows that $R_C(X^0)= R_E(X^0) = 0$, which contradicts H$_2$, which says $R_C=0$ implies $R_E>0$. Hence, \HHH is satisfied.
\end{proof}

\section{Trapping region and boundedness of trajectories} \label{Sec.TrappingRegion}
If we are given a differential equation or some trajectory that might be an H or H$^*$ trajectory, we will have to prove that it satisfies H$_B$, that is the trajectory is bounded, which might not be obvious. 
In applications, we will establish H$_B$ by showing in Prop.~\ref{Trap-Prop} that the following two alternative hypotheses together imply H$_B$:

\bigskip
\textbf{ H$_4$.}  (Whenever $B$ is too large, $B$ is decreasing.)\\{\it There exists  $B_4>0$}  such that for all $B \geq B_4$, 
$B' <0$. 

\bigskip 
\noindent 
If H$_4$ holds for a given value of $B_4$, it also holds for all larger values of $B_4$ \footnote{In HANDY, $B_4 \ge \lambda$, where $\lambda$ is the Environment's resource capacity; see Sec.~\ref{sec:HANDY}}. Hence we can always assume $B_4$ is chosen so that 
\begin{align}\label{B_4}
    B_4 \ge B(0),
\end{align}
the initial condition for $B(t)$.

\bigskip 
\textbf{ H$_5$.}  {(When the human population is too large, it is decreasing.)}\\{\it There exists $C_5>0$ (depending on $B_4$) such that if $B \le B_4$ and $C+E \ge C_5$, then $C', E' \le 0$.} 

\bigskip
\noindent
If H$_5$ holds for a given value of $C_5$, it also holds for all larger values of $C_5$. Hence we can always assume $C_5$ is chosen so that 
\begin{align}\label{C_5}
    C_5 \ge \max \{C(0), E(0)\}.
\end{align}

\begin{figure}
    \centering
    \includegraphics[width=0.6\textwidth]{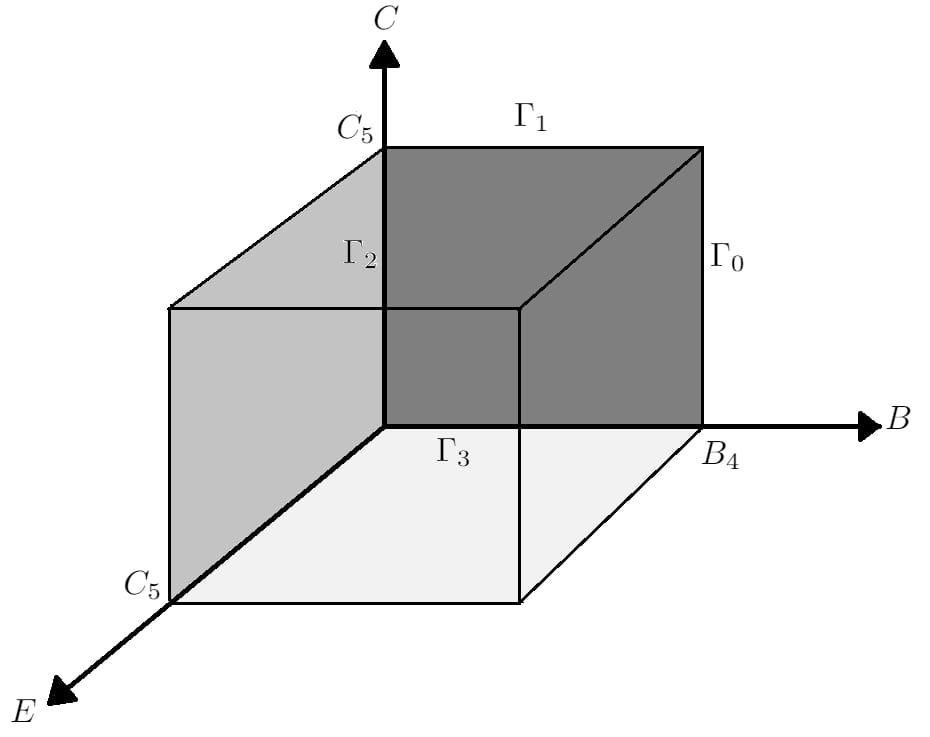}
    \caption{\textbf{{Constructing a trapping region $\Gamma$ for H$^*$.}}
    Given an initial point $X(0)$, $B_4$ and $C_5$ are chosen sufficiently large such that H$_4$ and H$_5$ are satisfied and \eqref{B_4} and \eqref{C_5} are satisfied: $X(0) \in \Gamma := [0, B_4] \times [0,  C_5] \times [0,  C_5]$. Then $\Gamma$ is a trapping region containing $X(t)$ for all $t \ge 0$.}
    \label{General-TrappingRegion}
\end{figure}

A \textbf{{trapping region $\Gamma$}} is a compact region such that for every trajectory, $X(\cdot)$, if $X(t_1) \in \Gamma$, then $X(t) \in \Gamma$ for all $t \ge t_1$, \citep{meiss2007differential}.

The following Proposition can help  establish H$_B$.

\begin{prop}[Boundedness of $B$, $C$, and $E$] \label{Trap-Prop}
Assume the trajectory $X(t)=(B,C,E)(t)$ satisfies H$^*_1$, H$_4$, and H$_5$.
Then there exists a trapping region $\Gamma$ containing $X(0) = (B, C, E)(0)$ in which the trajectory is bounded, so that H$_B$ is satisfied.
\end{prop}

\textbf{{Constructing the trapping region  $\Gamma$.}}\label{General-TrappingRegion-Arg}
Choose $B_4$ and $C_5$ according to H$_4$ and H$_5$, so that \eqref{B_4} and \eqref{C_5} are satisfied.
Then $\Gamma := [0,  B_4] \times [0,  C_5] \times [0,  C_5]$ contains $X(0)$.

\begin{proof}
Region $\Gamma$ has six faces. Three of the faces are determined by $B=0$, $C=0$ and $E=0$. Since all coordinates are positive, $X$ cannot leave through these three surfaces. 
Hyp.~H$_4$ implies the trajectory cannot leave $\Gamma$ through $B=  B_4$. 
Hyp.~H$_5$ similarly implies the trajectory cannot leave $\Gamma$ through $E =  C_5$ and $C =  C_5$.
Hence trajectories can not escape from $\Gamma$. 
Therefore, $\Gamma$ is a trapping region, each trajectory is bounded, and H$_B$ is satisfied.
\end{proof}

\section{The Elite-Dominated HANDY model and HANDY* model}\label{sec:HANDY}
In this section, we present what we call the ``Elite-dominated'' HANDY model. We believe our presentation and notation are simpler but the model is the same as the HANDY model in  \cite{motesharrei2014human} --- except for the addition of a decay rate, $\varepsilon$, for stored food, one restriction, and one generalization. We discuss these modifications after describing the model.

As in the previous sections, there are two human populations, Commoners, $C(t)$, and Elites, $E(t)$. 
We view the function $B_{Env}(t)$, $(t \ge 0)$ as the amount of wild or unharvested food --- both animals and crops (beans, berries, bunnies, buffalo, bluefish, etc.) --- available to the population. In HANDY, $B_{Env}$ denotes regenerating resources.
The amount of stored food is $B_{Stor}(t)$. 

\begin{equation}
        \begin{cases}
        B'_{Env} &= Q(B_{Env})- H, \\
        B'_{Stor} &= H -  F - {\varepsilon}B_{Stor}, \\        
        C'  &= G(Z) C , \\
        E'  &= G(\kappa Z) E .
        \end{cases}
        \label{4-d-model}
\end{equation}
where 
\begin{align}
    Q &:= \gamma B_{Env} \cdot (1 -  \frac{B_{Env}}{\lambda}), \ \ \  \mbox{(food reproduction rate)}, \label{Q} \\ 
    H &:= \nu B_{Env}C, \label{harvesting} \ \ \ \mbox{(rate of harvesting)}, \\
    Z &:= \frac{B_{Stor}/\rho }{C+\kappa E}, \label{Z} \ \ \ \mbox{(food supply / food demand)}, \\
    F &:= \sigma \min \{C+\kappa E, \frac{B_{Stor}}{\rho}\} \label{F} \\ 
    &=\sigma\cdot(C+\kappa E)\min\{1,Z\}, \ \ \ \mbox{(rate of food consumption)}, \nonumber\\
    G &:= \xi_1 + (\xi_2 - \xi_1)\min\{1,Z\}, \ \ \ \mbox{(Commoner per capita change rate)}. \label{G}
\end{align}

The model parameters are (these are kept constant in each HANDY scenario): 

\noindent $\nu>1$ (harvesting factor);

\noindent $\lambda>0$ (Environmental Resource Capacity, \textit{i.e.}, maximum capacity of food resource in the absence of people); 

\noindent $\gamma>0$ (maximum regeneration rate of environmental food); 

\noindent $\varepsilon>0$ (stored food decay rate);

\noindent $\sigma>0$ (food per capita needed for Commoners to attain maximum change rate);

\noindent $\rho^{-1}>0$ (maximum rate of stored food distribution);

\noindent $\kappa >1$ (Inequality factor: each Elite receives $\kappa$ times as much as a Commoner);

\noindent $\xi_1 <0 <\xi_2$ (minimum and maximum per capita change rates of people). 

\begin{figure} 
    \centering
    \includegraphics[width=0.75\textwidth]{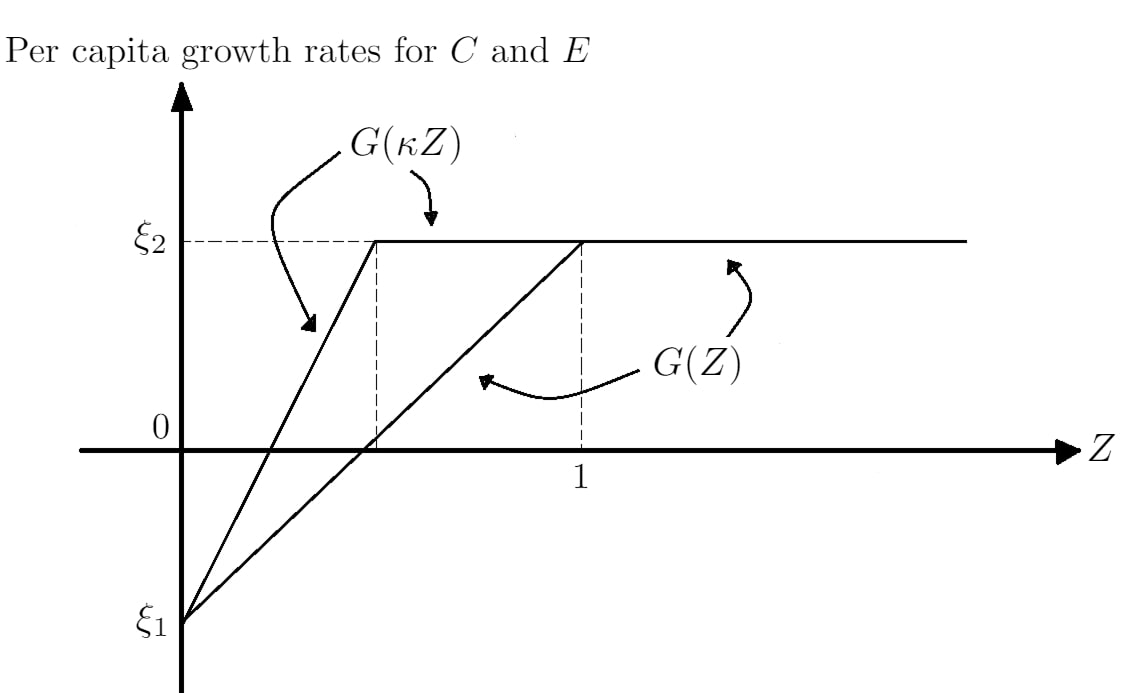}
    \caption{HANDY per capita change rate for Elite and Commoner populations. Compare with Fig.~\ref{Fig:rela}.}
   \label{Fig:growthratehandy}
\end{figure}
\skiplines{1}

By the \textbf{{Elite-dominated HANDY model}} we mean (\ref{4-d-model}) under the conditions that (\ref{params}) is satisfied:
\begin{align} \label{params}
    \lambda, \ \gamma, \ \varepsilon, \ \rho, \ \nu, \ \sigma>0, \  \kappa>1,   \mbox{ and} \ \xi_1<0<\xi_2.
\end{align}
Note that this Elite-dominated HANDY model satisfies assumptions \eqref{eq:h2} and \eqref{eq:h3}, in particular, $R_E \ge R_C$.

\textbf{{Our restriction of HANDY:}} We assume per capita food supply is a surrogate for health and reproduction rates.  
We consider only parameter sets that are what we call ``\textbf{Elite-dominated}''. 
That is, we exclude cases where at some time one population could be getting more food per capita than the other while having a lower per capita change rate. We achieve this restriction by making one change. We assume both populations have the same minimum and maximum change rates (reproduction minus mortality). 
In fact we could instead assume that there are two pairs of $\xi_i$, one for Elites and one for Commoners satisfying $\xi_1^C \le \xi_1^E<0< \xi_2^C\le \xi_2^E$; 
we would still have \HHH and H$_3$ satisfied and population collapse will still be true, 
but to keep notation simple, we do not pursue this path.

We generalize some aspects of HANDY below.

\textbf{{The HANDY$^*$ model (generalized Elite-dominated HANDY).}} 
\cite{motesharrei2014human} HANDY has a variety of constant parameters. Here we make them dependent on time with some restrictions. 
We allow most of these parameters in \eqref{params} to vary as if there are climatic and seasonal variations in weather, as well as other non-constant phenomena including disease. Hence, we turn these parameters into time-dependent functions.

For any function $f(t)$ we write $\boldsymbol{\inf f}$ for $\inf_{t \ge 0} f(t)$ and $\boldsymbol{\sup f}$ for $\sup_{t \ge 0} f(t)$. We assume~\eqref{4-d-model} is satisfied. Of the parameters in \eqref{params}, only $\sigma, \xi_1$, and $\xi_2$ must remain constant in HANDY$^*$. For HANDY$^*$, we assume the following;

\begin{equation}
        \begin{cases}
        \lambda, \gamma, \varepsilon, \rho, \nu, \kappa: [0, \infty) \to \mathbb{R} \ \mbox{are all bounded} \ C^1 \ \mbox{functions, and} \\
        \mbox{the absolute values of their derivatives are bounded.} \\        
        \inf \lambda, \ \inf \gamma, \ \inf \varepsilon, \ \inf \rho, \ \inf \nu >0, \ \inf \kappa>1, \\
        \sigma>0, \ \xi_1<0<\xi_2 \ \mbox{(constants)}. 
        \end{cases}
        \label{params2}
\end{equation}

We say $(B_{Env},B_{Stor},C,E)(t)$ is a \textbf{{HANDY$^*$ trajectory}} if it satisfies the Elite-dominated HANDY model, \eqref{4-d-model}-\eqref{G} plus \eqref{params2}.
\bigskip

\textbf{ How much food is distributed to the people?}
The function $F$ in~\eqref{F} is the rate at which food is taken out from the stored food $B_{Stor}$ and distributed to the people.
Each Elite always gets $\kappa$ times as much food as a Commoner.
There are two possible food distribution plans.
When food is plentiful, \textit{i.e.}, $Z \ge 1$, Commoners get just enough to maintain their maximal change rate, \textit{i.e.}, $\sigma$ per person. When  $Z<1$, each Commoner gets less food.
Then only the amount of food $\frac{B_{Stor}}{\rho}$ is allocated, and this amount is distributed on a per capita basis --- with each Elite getting $\kappa$ times as much as a Commoner.
Then the Commoner change rate decreases and can even become negative. 

We will write
\begin{align}\label{eq:B}
     B := B_{Env} + B_{Stor}, \ \ \ \mbox{(total food).}
\end{align}

\begin{thm} [Every HANDY$^*$ trajectory with strictly positive coordinates is an H$^*$ trajectory. Hence, it has population collapse.]
Let $(B_{Env},B_{Stor},C,E)(t)$ be a HANDY$^*$ trajectory.
Define $B$ in \eqref{eq:B}. Assume  $B(0), C(0), E(0)>0$. 
Then $(B,C,E)(t)$ is an H$^*$ trajectory; hence $C(t),E(t)\to 0$ as $t \to \infty$.
\label{hyps-4-dim}
\end{thm}
The proof of this result comes after next lemma.

\bigskip

\textbf{{The flow of biomass in the HANDY$^*$ model.}} People consume stored food and stored food decays. 
A unit of food can result in at most $\frac{\sigma}{\xi_2 - \xi_1}$ additional people. (Recall $\sigma, \ \xi_1$, and $\xi_2$ are constants.) 
Stored food is replenished by harvesting at the rate 
$$H(t):= \nu(t) \cdot B_{Env}(t)\cdot C(t).$$
Given a HANDY$^*$ trajectory, write 

\begin{equation} 
        \begin{cases}
        Y := (C+E) \frac{\sigma}{\xi_2 - \xi_1}+ B_{Stor}; \\
        \hat \varepsilon := \min \{|\xi_1|, \inf \varepsilon \}.       
        \end{cases}
        \label{twodefs}
\end{equation}
Notice that  $\hat \varepsilon >0$ is
the minimum of the human population decay rates in the absence of stored food $-\xi_1>0$ and the minimum stored food decay rate $\inf \varepsilon>0$.
We will obtain 
\begin{align}
\label{main-ineq}
    Y' \le H - \hat \varepsilon \cdot Y.
\end{align}
This implies that if $H(t) \to 0$, then $Y(t) \to 0$  as $t \to \infty$, and we obtain these  conclusions: $B_{Stor}(t), E(t)$ and $C(t) \to 0\mbox{ as }t\to \infty$. We will use the following lemma to prove H$_Z$ is satisfied.

\begin{lem} [Three birds with one stone] \label{birds}
Let $Y(t)$ and $\hat \varepsilon$ satisfy \eqref{twodefs}. 
Then Ineq.~\eqref{main-ineq} is satisfied by every HANDY$^*$ trajectory, and if $H(t) \to 0\mbox{ as }t\to \infty$, then 
\begin{align}
    B_{Stor}(t) \to 0, \ E(t) \to 0, \mbox{ and }C(t) \to 0\mbox{ as }t\to \infty.
\end{align}
\end{lem}

\begin{proof}
From \eqref{G},
$G(Z) = \xi_1 + (\xi_2 - \xi_1)\min\{1,Z\}$. Therefore, we will use the fact that for $\kappa >1$, 
\begin{align}\label{minZ}
    \min\{1,\kappa Z\} \le \min\{\kappa ,\kappa Z\} = \kappa \min\{1,Z\},
\end{align}
for all $Z \ge 0$. Hence,
\begin{align}
    E' &= G(\kappa Z)E = \xi_1E +  (\xi_2 - \xi_1) \min\{1,\kappa Z\} E, \label{E'}\\
    C' &= G(Z) C = \xi_1C +  (\xi_2 - \xi_1) \min\{1,Z\} C. \label{C'}
\end{align}
Therefore, 
\begin{align}
    C'+E' &= (C+E)\xi_1 + (\xi_2 - \xi_1) \big(\min\{1,Z\} C +\min\{1,\kappa Z\} E\big) \nonumber\\
    &\le (C+E)\xi_1 + (\xi_2 - \xi_1) \big(\min\{1,Z\} (C + \kappa E)\big), \ \mbox{ (from \eqref{minZ})}.\nonumber
\end{align}
Define $S:=\frac{\sigma}{(\xi_2 - \xi_1)}(C+E)$. Then
\begin{align}
    S' = \frac{\sigma}{(\xi_2 - \xi_1)}(C'+E') &\le \xi_1 S + \sigma \cdot (C+ \kappa E) \min\{1,Z\} \nonumber\\
    &=\xi_1 S + F, \ \mbox{(from \eqref{F})}. \nonumber
\end{align}
We will add the above inequality to the following,  (from \eqref{4-d-model}).
\begin{equation}
    B_{Stor}' = H -  F -\varepsilon \cdot B_{Stor};\nonumber
\end{equation}
Hence 
\begin{align}
    (B_{Stor}+S)' &= H + \xi_1 S - \varepsilon\cdot B_{Stor}\nonumber \\
      &\le H  -\hat \varepsilon \cdot(B_{Stor}+S), \ \mbox{($\hat \varepsilon$ is from \eqref{twodefs})}. \nonumber
\end{align}
Define $Y := B_{Stor}+S$. Then,
\begin{align}
    Y' \le H - \hat \varepsilon Y.
\end{align}
Now, if $H(t) \to 0$ as $t \to \infty$, then $Y(t) \to 0$ which implies $B_{Stor}(t), C(t)$, and $E(t) \to 0$ as $t \to \infty$.  
\end{proof}

\bigskip

{\textbf{Note.}} The reader may wonder why we chose to have H$^*_1$ require that $\frac{C'}{C}$ and $\frac{E'}{E}$ are uniformly continuous.
The reason is that the functions used in defining the HANDY differential equations use $\min \{1, Z\}$, which is uniformly continuous but not differentiable. 
Hence the right-hand sides of \eqref{4-d-model} are only piecewise smooth. 
Hence we invoke uniform continuity of $\frac{C'}{C}(t)$ and $\frac{E'}{E}(t)$.

\skiplines{1}

Let $\zeta$ be defined so that $G(\zeta)=0$; hence
\begin{equation} \label{eta}
    \zeta := \frac{-\xi_1}{\xi_2 - \xi_1} \ \mbox{, which is in} \ (0,1) \ \ (\mbox{from} \ \eqref{params}).
\end{equation}

\begin{proof}[\bf Proof of Theorem \ref{hyps-4-dim}]
We now show the HANDY$^*$ trajectory $X :=(B,C, E)(t)$ in Thm.~\ref{hyps-4-dim} satisfies H$^*$ in the following order: H$^*_1$, H$_B$, H$^*_2$, H$_3$, and H$_Z$.

\textbf{{(HANDY$^*$ $\Longrightarrow$ H$^*_1$)}}. 
HANDY$^*$ trajectories are defined for all time since the right-hand-sides of \eqref{4-d-model} and \eqref{eq:B} grow at most linearly in $|B_{Env}|, |B_{Stor}|, |C|$ and $|E|$. 

The parameters in \eqref{params2} are uniformly continuous since their derivatives' absolute values are bounded. The assumptions in \eqref{params2} imply that if the trajectory is bounded, then $\sup_{t \ge 0}|X'(t)|< \infty$, and $\frac{C'}{C}$ and $\frac{E'}{E}$ are uniformly continuous. Hence H$_1^*$ is satisfied.

\textbf{{(HANDY$^*$ $\Longrightarrow$ H$_B$)}}. Next, we prove H$_4$ and H$_5$, which together by Prop.~\ref{Trap-Prop} imply every trajectory is bounded and so H$_B$ is satisfied.

\textbf{{(HANDY$^*$ $\Longrightarrow$ H$_4$)}}. 
By (\ref{params2}), $\sup_{t \ge 0}\frac{\gamma \lambda^2}{\varepsilon}>0$ and $\sup \lambda>0$. Choose 
\begin{align} \label{Eq:B_4}
     B_4 \ge \text{max} \{\sup_{t \ge 0}\frac{\gamma \lambda^2}{\varepsilon}, 4\sup_{t \ge 0}\lambda \}.
\end{align}
Assume $B \ge B_4$.
Since $B= B_{Env}+B_{Stor}$, one of the following two cases hold.

\textbf{{Case1.}} Suppose $B_{Stor} \ge \frac{B}{2}$. 
From the identity $\max \{x(1- \frac{x}{\lambda})\} = (\frac{\lambda}{2})^2$ for $\lambda>0$,
we always have $B_{Env} \cdot       (1- \frac{B_{Env}}{\lambda}) \le (\frac{\lambda}{2})^2$.
Therefore,
\begin{align}
    B' &= \gamma  B_{Env}\cdot(1 - \frac{B_{Env}}{\lambda}) - F - \varepsilon B_{Stor}\nonumber \\
    &\le \gamma  B_{Env}\cdot(1 - \frac{B_{Env}}{\lambda}) - \varepsilon B_{Stor}  \nonumber \\
    &\le \gamma (\frac{\lambda}{2})^2 - \varepsilon B_{Stor} \le \frac{\varepsilon}{4} \cdot (\frac{\gamma\lambda^2}{ \varepsilon} - 2B) \le \frac{\varepsilon}{4} \cdot (\sup \frac{\gamma\lambda^2}{ \varepsilon} - 2B) \le - \frac{\varepsilon}{4} B<0. \nonumber 
\end{align}

\textbf{{Case2.}} Otherwise $B_{Env} \ge \frac{B}{2}$. By \eqref{Eq:B_4}, $B \ge B_4 \ge 4 \sup \lambda$. Hence, $B_{Env} \ge 2 \sup \lambda \ge 2\lambda$. Therefore, $1- \frac{B_{Env}}{\lambda} \le -1$. Hence
\begin{align}
    \frac{B'}{B} &\le \gamma  \frac{B_{Env}}{B}\cdot(1- \frac{B_{Env}}{\lambda}) - \varepsilon B_{Stor} \\
    &\le \gamma  \frac{B_{Env}}{B}(-1) \le -\frac{\gamma}{2} <0. \nonumber 
\end{align}
Hence when \eqref{Eq:B_4} is satisfied, H$_4$ is true. 

\textbf{{(HANDY$^*$ $\Longrightarrow$ H$_5$)}}.
We need to prove there exists $C_5>0$ (depending on $B_4$) such that if $B \le B_4$ and $C+E \ge C_5$, then $\frac{C'}{C}, \frac{E'}{E}\le 0$.

Inequality $ \frac{E'}{E} = G(\kappa Z)\le 0$ is satisfied iff $Z \le \frac{\zeta}{\kappa} =: Z_0$, where $\zeta$ is defined in \eqref{eta}; and $G(\kappa Z_0) = 0$. When $Z \le Z_0$ holds, $\frac{C'}{C} = G(Z) < G(\kappa Z) \le 0$.
We want $C_5$ sufficiently large that $Z \le Z_0$. 

Since $B_{Stor} \le B \le B_4$, 
\begin{align}
    Z = \frac{B_{Stor}}{\rho (C + \kappa E)}  \le \frac{B_4}{\rho C_5}. \nonumber
\end{align}
We want the right-hand side to be $\le Z_0$ which equals  $\frac{\zeta}{\kappa}$.
Hence choosing $C_5 > \frac{\kappa B_4}{\zeta \rho}$ makes Hyp.~ H$_5$ true. 

Hence H$_B$ is satisfied.

\textbf{{(HANDY$^*$ $\Longrightarrow$ H$_2^*$)}}.
Let $\varepsilon_2^* := (\inf \frac{\kappa-1}{\kappa+2})(-\xi_1)>0$. If $|\frac{C'}{C}|< \delta_2^*$ (\textit{i.e.}, $-\delta_2^* < G(Z) < \delta_2^*$), then by Eqs.(\ref{harvesting})-(\ref{G}),
\begin{align}
    \frac{E'}{E} = G(\kappa Z) > \xi_2 + (\kappa -1)(\xi_2-\xi_1) - \kappa \cdot (\delta_2^* + \xi_2).
\end{align}
Therefore, 
\begin{align}
    \frac{E'}{E} - \frac{C'}{C} > (\kappa -1)(-\xi_1) - (\kappa+1)\delta_2^* &> \frac{\kappa-1}{\kappa+2}(-\xi_1)\nonumber \\
    &\ge (\inf \frac{\kappa-1}{\kappa+2})(-\xi_1). \nonumber
\end{align}

\textbf{{(HANDY$^*$ $\Longrightarrow$ H$_3$)}}. 
By \eqref{4-d-model}, \eqref{G} and \eqref{params}, since $\kappa > 1$,

\begin{align} \label{sameHyp3}
      \frac{E'}{E} = G(\kappa Z) \ge G(Z) =\frac{C'}{C}.
\end{align}

\textbf{{(HANDY$^*$ $\Longrightarrow$ H$_Z$)}}.  If $C(t) \to 0$ as $t \to \infty$ and $B_{Env}$ is bounded, then $H(t)\to 0$ as $t \to \infty$.
By Lemma~(\ref{birds}),
$E(t) \to 0$ as $t \to \infty$. 
\end{proof}

\section{Lyapunov Function}\label{Lyapunov}
Consider a differential equation on a finite dimensional linear space $\mathcal E$, 
\begin{equation}\label{DE}
    X' = F(X), 
\end{equation}
where $F : \Omega_F \to \mathcal E$ is a $C^1$ function, where $\Omega_F \subset \mathcal E$ is closed set.
Assume $\mathbf V$ is a $C^1$ real-valued function that is defined on at least the part of the domain of $F$.
Define $\mathbf{\dot V}(X):=\frac{d}{dt}\mathbf V(X(t))|_{t=0} = \mbox{grad} ~\mathbf{V}(X) \cdot F(X)$.
We refer to $\mathbf V$ as a Lyapunov function if $\mathbf{\dot V} \le 0$ on the domain of $\mathbf V$.

Define $L^+ := L^+(X(t))$ or $L^+(X)$ to be the positive limit set of $X(t)$ as $t \to \infty$.
The set $L^+$ is invariant, \ie, a set $S$ is \textbf{invariant} if whenever $Y(t)$ is a trajectory and $Y(0) \in S$, then $Y(t) \in S$ for all $t \in \mathbb R$.
If $X(t)$ is bounded, then $L^+$ is compact.
\skiplines{1}

\noindent
{\textbf{Definition.}}
We say a trajectory $Y(\cdot)$ is \textbf{{doubly bounded}} if $Y(t)$ is defined for all $t \in \mathbb R$ and if there is a constant $\beta_Y>0$ such that $|Y(t)| \le \beta_Y$ for all $t \in \mathbb R$.

\begin{prop}[\bf {Generalized Barbashin--Krasovskii--LaSalle  (BKL) Theorem}] \label{Barbashin-Krasovskii-LaSalle}
Assume $ \Omega_F$ is a closed set and let
$F: \Omega_F \rightarrow \mathbb{R}^n$ be a $C^1$ function. 

Let $\Omega_\mathbf{V}\subset\Omega_F$ be an invariant set and $\mathbf{V}:  \Omega_\mathbf{V}\to \mathbb{R}$ be $C^1$.
Assume $\mathbf{\dot V}\le 0 \mbox{ on }
\Omega_\mathbf V$. Let $D := \{p:\mathbf{\dot V}(p) = 0\}$.
Let $X(t)$ be a bounded solution of (\ref{DE}) that is in $\Omega_{\mathbf V}$ for all $t\ge 0$.
 
\textbf{{(BKL1)}} If there are no doubly bounded solutions in $D$, then $X(t) \to \mathbb E_0$ where $\mathbb E_0:=\Omega_F \backslash \Omega_\mathbf{V}$ as $t\to\infty$. 

\textbf{{(BKL2)}} Let $Y(t)\in L^+$ be a solution of (\ref{DE}).  Then $Y$ is doubly bounded. If  $Y(0) \in L^+ \cap \Omega_{\mathbf V}$, then $Y(t) \in D$ for all $t \in \mathbb R$.
\end{prop}

Our proof of Thm.~{\ref{mainThm}} uses the conclusion (BKL1).
Conclusion (BKL2) is a version of the standard BKL Theorem.  See \cite{haddad2011nonlinear}, p. 147, for a standard version where $\Omega_\mathbf V$ is compact. 

 \begin{proof}
Let $X(t)$, for $t \ge 0$, be a solution of \eqref{DE} with $X(0) \in \Omega_{\mathbf V}$.
Since $\mathbf{\dot V} \le 0$, $\mathbf{V}(X(t))$ is a non-increasing function of time.  
For each $r_0 \in L^+(X(0))$, there is an increasing sequence $t_n \to \infty$ as $n \to \infty$, such that $X(t_n) \to r_0$ as $n \to \infty$.
Let $Y(t)$ be the solution for which $Y(0)=r_0$.

If $r_0 \in \Omega_{\mathbf V}$, by continuity of $\mathbf V$ on $\Omega_{\mathbf V}$, 
\begin{align}
  \mathbf V(r_0) = \mathbf V (\lim_{n \to \infty} X(t_n)) = \lim_{n \to \infty}\mathbf V(X(t_n)).  
\end{align}
So $Y(t) = \lim X(t_n+t)$. Then
\begin{align}
    \mathbf V(Y(t)) = \lim_{n \to \infty} \mathbf{V}(X(t_n+t)) = \mathbf{V}(r_0).
\end{align}
Therefore, $\mathbf{\dot V}(Y(t))=0$. 
Thus, $L^+(Y(t)) \subset D$.
If there are no solutions lying in $D$, then there is no such $r_0$ in $\Omega_V$ and $L^+(X(t)) \subset \mathbb E_0$ where $\mathbb E_0 = \Omega \backslash \Omega_{\mathbf V}$, proving (BKL$1$). 
Since $L^+$ is compact $Y$ is doubly bounded. Then (BKL$2$) follows from the fact that $L^+ \cap \Omega_{\mathbf V}$ is invariant.
\end{proof}
\section{Proof of theorem~{\ref{mainThm}} \ (H \texorpdfstring{$\Longrightarrow$} \ \ collapse)}\label{theorem2}

Let $X(t)=(B, C, E)(t)$ be an H trajectory. By H$_1$, $E(t)>0$. 
To prove Thm.~\ref{mainThm}, first, we show in Lemma~{\ref{Ratio-General}} that the ratio $\frac{C}{E}$ of the populations is never increasing.
Next, we define a compact trapping region $\Gamma$ that contains $X(0)$, so $X(t)$ remains in $\Gamma$ for all $t>0$.
Hence $X(t)$ is bounded. 
So our version of the Barbashin--Krasovskii--LaSalle (BKL) Theorem (Prop.~\ref{Barbashin-Krasovskii-LaSalle}) can use the boundedness of the trajectory $X(t)$. 

The Lyapunov function in the proof of Thm.~\ref{mainThm}, $\mathbf{V} :=\frac{C}{E}$, is not defined on some of the domain $\Omega$ of the differential equation (\textit{i.e.}, that is, when $E=0$). 
This fact is critical because in our application it will become apparent that the limit set of each trajectory lies in $E=0$, which is where $\mathbf V$ is not defined. 

\begin{lem}\label{Ratio-General}
Let $X(t) = (B, C, E)(t)$ be a trajectory satisfying H$^*_1$ and H$_3$. Let $\mathbf{V} =\frac{C}{E}$ for $E>0$. Then $\mathbf{\dot V}\le 0$. 
\end{lem}
\begin{proof}
Recall $E>0$, $C>0$. Hence
\begin{align}
    \mathbf{\dot V}&= \frac{C'E - E'C}{E^2} = (\frac{C'}{C} - \frac{E'}{E}) \frac{C}{E}.
\end{align}
By H$_3$, $\frac{C'}{C} - \frac{E'}{E} \leq 0$. Hence $\mathbf{\dot V} \le 0$.
\end{proof}

\begin{proof}[{\textbf{Proof of Thm.~\ref{mainThm}}}] 
Claim.
There are no doubly bounded trajectories that stay in $D$ for all time.
If the claim is true, then (BKL1) implies $X(t) \to \mathbb E_0 := \{(B, C, E): E=0\}$, the set where $\mathbf V$ is not defined. 
Hence $E(t) \to 0$ as $t \to \infty$. 
Since $\mathbf V$ is decreasing $\frac{C}{E}(t)$ is bounded, so $C(t) \to 0$ as $t \to \infty$, which would prove Thm.~\ref{mainThm}.

Suppose the claim is false; that is, there is a doubly bounded trajectory $Y(t) = (B_Y, C_Y, E_Y)(t)$ in $D$ for all $t \in \mathbb R$.
From the definition of $D$, $\mathbf{\dot V}(Y) = 0 = \frac{C_Y'}{C_Y} - \frac{E_Y'}{E_Y}$. 
Since H$_2$ says there are no points where $C'_Y = 0 = E'_Y$, it follows that $C_Y(t)$ and $E_Y(t)$ are both monotonic increasing or both are monotonic decreasing.

Suppose $C_Y(t)$ and $E_Y(t)$ are monotonic increasing.
Since $Y(t)$ is bounded, there is a limit point $Z^0$ of $Y(t)$ as $t \to \infty$.
Then
\begin{align}\label{R_c=R_E}
    R_C(Z^0) = 0 = R_E(Z^0).
\end{align}
Since no such point exists, this contradicts our assumption that $C_Y$ and $E_Y$ are increasing.

Hence $C_Y$ and $E_Y$ must be monotone decreasing.
Now let $Z^0$ be a limit point of $Y(t)$ as $t \to -\infty$, and we again get \eqref{R_c=R_E}. 
Hence there are no doubly bounded trajectories in $D$, proving the claim and completing the proof. 
\end{proof}

Figure~\ref{Fig:rela} shows a situation where $\mathbf{\dot V}$ can be zero on a set that is 
possibly a large fraction of the space.
\section{Proof of Theorem~\ref{mainThmstar} \ (H\texorpdfstring{$^*$}
\texorpdfstring{$\Longrightarrow$} collapse)}\label{proofmainThmstar}
We introduce a powerful implication of H$_2^*$.

\bigskip
\begin{itemize}
    \item [\textbf{{\hh.}}] (Whenever $C'(t)$ is near $0$, $E'(t)>0$, and then, $\frac{C}{E}$ decreases by at least some fixed fraction.)\\ {\it There exist $\delta_2 >0$ and $\ \varepsilon_2>0$ such that $|C'(t)|<\delta_2$ implies
  \begin{align}
      \frac{C}{E}(t+1) < (1+\varepsilon_2)^{-1} \cdot  \frac{C}{E}(t). 
  \end{align}
  Note that $\delta_2 \ \mbox{and} \ \varepsilon_2$ can depend on the trajectory.}
\end{itemize}
\skiplines{1}

\begin{lem}\label{lem_star}
H$_1^*$, H$_2^*$, and H$_3$ together imply \hh.
\end{lem}

\begin{proof} 
Let $\delta = \delta_2^*$ be as in \HHH.
By uniform continuity of $\frac{C'}{C}$ (from H$^*_1$),  there is a time $\tau$, $(0<\tau \le 1)$, such that if $|\frac{C'}{C}(t)|< \frac{\delta}{2}$, then $\frac{C'}{C}\le \delta$ for $\delta \in [t, t+\tau]$.
Write $\mathbf V :=\frac{C}{E}$. Write $\mathbf V'$ for $\frac{d}{dt}\mathbf V$.
During that time $(t, t+ \tau)$, $\frac{\mathbf{ V'}}{\mathbf V}=\frac{C'}{C}-\frac{E'}{E} \le - \delta$ by H$^*_2$. During that time $\log \mathbf V$ decreases, 

\begin{align*}
    \log \mathbf V(t+\tau) - \log \mathbf V(t) = \int_{t}^{t+\tau} \frac{\mathbf{V'}}{\mathbf V} dt < \int_{t}^{t+\tau} - \delta dt   < - \tau \delta. 
\end{align*}
Hence $\frac{\mathbf V(t + \tau)}{\mathbf V(t)}< e^{-\tau \delta}$. Let $\varepsilon_2 := \tau \delta >0$. Since $e^{\varepsilon_2} \ge 1+ \varepsilon_2$ and $\tau<1$, we obtain
\begin{align*}
    \mathbf V(t) \ge e^{\tau \delta} \mathbf V(t+\tau) \ge (1+\varepsilon_2)\mathbf V(t+1),
\end{align*}
so \hh \ is satisfied with this $\varepsilon_2$. 
\end{proof}

\begin{proof}[Proof of Thm.~\ref{mainThmstar}] 
Each H$^*$ trajectory $(B, C, E)(t)$ is bounded for $t \ge 0$. 
We prove $C(t) \to 0$. 
We will split the task into two cases.

\textbf{{Case 1.}} There exist $\delta>0$, $T >0$ such that $|\frac{C'}{C}(t)|> \delta$ for all $t \ge T$. 

If $\frac{C'}{C}> \delta$ then $C(t) \to \infty$ as $t \to \infty$ contradicting boundedness. So that can not occur. If $\frac{C'}{C}< - \delta$, then $C(t) \to 0$ as $t \to \infty$.

If instead, there exists no such $\delta$ and $T$, then there following holds. 

\textbf{{Case 2.}} 
By Lemma \ref{lem_star}, \hh \ is true. 
Let $\delta_2$ and $\varepsilon_2$ be the values in \hh. There exists $t_n \to \infty$ for $n = 0, 1, 2, \cdots$ for which $|\frac{C'}{C}(t_n)| \le \delta_2$. Without loss of generality we can assume $t_{n+1} > t_n + 1$. 
From \hh \ and since $\frac{C}{E}(t)$ is monotone decreasing,
\begin{align}
    \frac{C}{E}(t_n) \le (1+\varepsilon_2)^{-n}\frac{C}{E}(t_0).
\end{align}
Hence $\frac{C}{E}(t) \to 0$
as $t_n \to \infty$. Since $E(t)$ is bounded, $C(t) \to 0$ as $t \to \infty$.  
Then Hyp.~H$_Z$ implies $E(t) \to 0$ as $t \to \infty$.
\end{proof}
\section{Why don't all societies collapse? Downward Mobility?}\label{no-Collapse}
We have shown that for our H and H$^*$ models, population collapse always occurs if $E(0)>0$. But not all populations on earth collapse. Our results thus begin a conversation about how actual societies avoid collapse caused by Elite dominance. Many societies limit the size of the Elites through a process of primogeniture, in which the oldest male child in an Elite family inherits the wealth of the family. 
\begin{figure}
    \centering
    \includegraphics[width=0.75\textwidth]{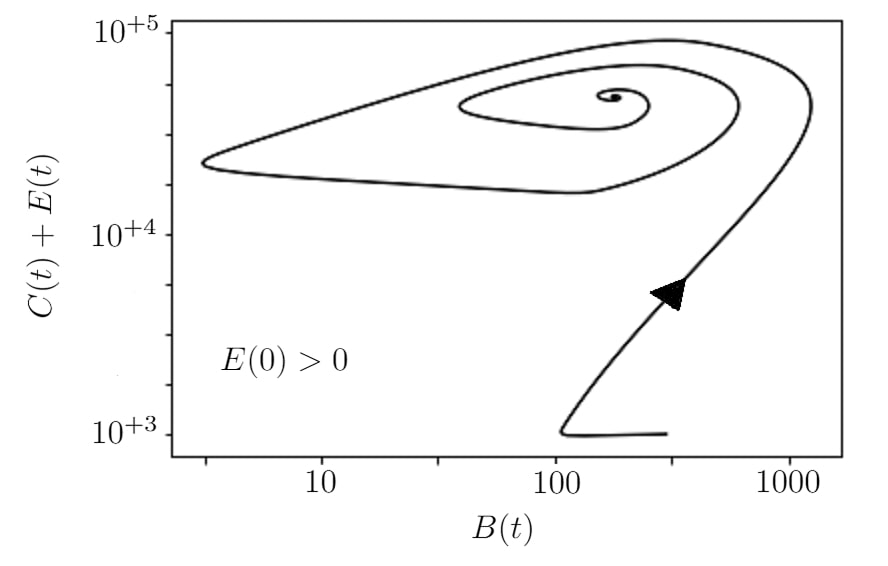}
    \caption{ {\textbf{{Equilibrium of Elite-dominated HANDY Model with downward mobility.}} A trajectory is calculated for initial conditions and parameters in Table~\ref{table-params}}.}
    \label{fig:movementModel}
\end{figure}
In this section we represent a new version of H model. We show that the mobility of Elites to Commoners population sometimes create a stable society with an equilibrium state. 

The \textbf{Elite-dominated HANDY model with downward mobility} is:
\begin{equation} 
        \begin{cases}
        B_{Env}' &= Q(B_{Env}) - H , \\
        B_{Stor}' &= H - F - \varepsilon B_{Stor}, \\
        C'  &= G(Z) \cdot C + \mu E^2,   \\
        E'  &= G(\kappa Z) \cdot E - \mu E^2,   
        \end{cases}
        \label{Handymobility}
\end{equation}
in which $\mu>0$ is a constant factor of mobility and $B = B_{Env} + B_{Stor}$.

The \textbf{{equilibrium}} is given when $B_{Env}'= B_{Stor}' = C' = E'=0$.
Therefore, 
\begin{align}
    Q(B_{Env}) &= H = F+ \varepsilon B_{Stor}, \label{equi:H}\\
    G(Z) \cdot C &+ G(\kappa Z) \cdot E =0. \label{equi:G}
\end{align}
Suppose $Z \ge 1$. By Eqs.~\eqref{G} and \eqref{equi:G}, $\xi_2 C + \xi_2 E =0$. But $\xi_2, C$ and $E>0$. Therefore, for $Z \ge 1$, $\xi_2 C + \xi_2 E \neq 0$. 

Now, let $Z <1$. We only consider the case where $\kappa Z \le 1$. Therefore,
\begin{align}
    F &= \frac{\sigma}{\rho} B_{Stor}, \label{equi:F} \\
    G &= \xi_1 + \frac{\xi_2 - \xi_1}{\rho} \cdot Z, \\
    H &= Q(B_{Env}) = (\frac{\sigma}{\rho} + \varepsilon) B_{Stor}.
\end{align}
The following definitions make the expressions of equilibria simpler:
\begin{align}
    L_1 &:= \frac{\gamma \xi_1 (\xi_2 - \xi_1)(\kappa \sigma +1)}{\sigma + \varepsilon \rho}, \label{L1} \\
    L_2 &:= \frac{\gamma \xi_1 \mu}{\nu} \label{L2}, \\ 
    L_3 &:= \Big((\lambda L_1)^2 + L_2^2 - 6\lambda L_1 L_2 - 4 \lambda \xi_1^2 L_1\Big)^\frac{1}{2}/ L_1.  \label{L3}
\end{align}
By Eqs.~\eqref{equi:H}-\eqref{L3}, the \textbf{ equilibrium values} $\boldsymbol{(B^{e}_{Env}, B^{e}_{Stor}, C^{e}, E^{e})}$, for model \eqref{4-d-model}, 
\begin{equation}
    \begin{cases}
    B^{e}_{Env} &= \frac{-L_3 - (\lambda L_1 - L_2)}{2L_1}, \\
    B^{e}_{Stor} &= \frac{\gamma \rho}{\sigma + \varepsilon \rho} B^{e}_{Env}\cdot (1- \frac{B^{e}_{Env}}{\lambda}), \\
    C^{e} &= \frac{\gamma}{\nu} \cdot (1 - \frac{B_{Env}^{e}}{\lambda}), \\
    E^{e} &= \frac{\xi_1 - \xi_2}{\xi_1}B_{Stor}^{e} - C^{e}.
    \end{cases}
\end{equation}

\cite{diamond2005collapse} describes the dichotomy between Elites and Commoners in Eastern Island. 

“As elsewhere in Polynesia, traditional Easter Island society was divided into chiefs and Commoners. To archaeologists today, the difference is obvious from remains of the different houses of the two groups. Chiefs [Elites] and members of the Elite lived in houses termed \textit{hare paenga}, $\ldots$. In contrast, houses of Commoners were relegated to locations farther inland, were smaller, and were associated each with its own chicken house, oven, stone garden circle, and garbage pit—utilitarian structures banned by religious \textit{tapu} from the coastal zone containing the platforms and the beautiful \textit{hare paenga}.”

\bigskip

\section{Discussion}
Models of physical and biological phenomena may capture key features without the full complexity of reality. Such models can have simplifying components and approximations that are not justifiable, all in the hope of better understanding the original phenomena. Can we determine what aspects of the model are responsible for qualitative phenomena seen in model solutions? Might they be due to the approximations?

Here we investigate models for one type of human society where the human population collapses. 
Our project began with the differential equations model, HANDY, and finally became a qualitative approach in which five hypotheses that are satisfied by 
an Elite-dominated HANDY model, in which Elites' population change rate is larger than or equal to that of Commoners', guarantee population collapse. 

Our interest is in the causes of Elite-dominated collapse, and we use HANDY as a door through which we can approach the question.

Furthermore, these hypotheses are satisfied by much more general situations such as our HANDY$^*$ model. This more general model allows many smooth fluctuations in the parametric functions in assumption \eqref{params2}.

The assumptions are in essence uniform smoothness of a bounded trajectory $X(t)$ in an Elite-dominated society for which there exist $\varepsilon_2^*>0$ and $\delta_2^*>0$ such that $|\frac{C'}{C}| \le \delta_2^*$ implies $\frac{E'}{E} - \frac{C'}{C}>\varepsilon_2^*$.

Perhaps our five H$^*$ hypotheses could be simplified or generalized a bit, but none can be simply eliminated. Two examples follow. These examples could be quite different from HANDY$^*$.

If Hyp. H$_B$ was eliminated, then models could be constructed satisfying the remaining four in which $C(t)$ and $E(t) \to \infty$ as $t \to \infty$ as hold for example with the following system.
$$\frac{C'}{C} = \frac{E'}{E} = 1; \ B' = 0.$$

If Hyp. H$_Z$ was eliminated, then models could be constructed satisfying the remaining four in which $C(t)\to 0$ and $E(t) \to \infty$ as $t \to \infty$. Imagine for example if the Elites had an unlimited external food source.

If we only assume H$^*_1$, \HHHH, and H$_3$, we can still conclude that for each bounded trajectory, $C(t)\to\infty.$ But we would not know if any non zero trajectories of HANDY are bounded.

In Sec.~\ref{no-Collapse}, we give social downward mobility as an alternative to population collapse.
Introducing terms into HANDY in which some Elites become Commoners prevents the ratio $\mathbf V = \frac{C}{E}(t)$
from becoming unsustainably small.

We hope that many readers will find that our approach clarifies why population collapse can occur and perhaps suggests how it can be avoided.

\skiplines{3}

{\textbf{Acknowledgement. }}
We thank Safa Motesharrei and Jorge Rivas for their comments that improved the manuscript. 

\newpage
\begin{table}[H] 
\centering
\begin{tabular}{ |l|l|l| }
  \hline
  \rowcolor{Gray}
  \textbf{{Parameter symbol}} & {\textbf{{Parameter name}}} & {\textbf{{Typical value(s)}}} \\
  \hline
  $\xi_1$ & minimum per capita change rate & $-4 \times 10^{-2}$ \\
  \hline
  $\xi_2$ & maximum per capita change rate & $2 \times 10^{-2}$ \\
  \hline
  $\nu$ & harvesting factor & $1.67 \times 10^{-5}$ \\
  \hline
  $\gamma$ &  maximum regeneration rate of & $1 \times 10^{-2}$  \\
  & Environmental food & \\
  \hline
  $\sigma$ & subsistence food per capita & $5 \times 10^{-4}$ \\
  \hline
  $\rho^{-1}$ & maximum rate of food distribution & $5 \times 10^{-3}$ \\
  \hline
  $\lambda$ & resource capacity of Environment & $1 \times 10^{+2}$ \\
  \hline
  $\kappa$ & inequality factor &  $1.5$ \\
  \hline
  $\varepsilon$ & Stored food decay rate & $1 \times 10^{-5}$ \\
  \hline
  $\mu$ & Elites mobility factor & $1 \times 10^{-4}$  \\
  \hline
  \hline
  \rowcolor{Gray}
  \textbf{{Variable symbol}} & {\textbf{{Variable name}}} & {\textbf{{Typical initial}}} \\
  \rowcolor{Gray}
  & & \textbf{{value(s)}} \\
  \hline
  $B$ & total food resources & $3 \times 10^{+2}$  \\
  \hline
  $B_{Env}$ & Environmental food & $3 \times 10^{+2}$ \\
  \hline
  $B_{Stor}$ & Stored food & $0$ \\
  \hline
  $C$ & Commoner population & $1 \times 10^{+3}$  \\
  \hline
  $E$ & Elite population & $0$ and $1$  \\
  \hline
  \hline
  \rowcolor{Gray}
  \textbf{{Variable symbol}} & \textbf{{Variable name}} & \textbf{{Defining}} \\
  \rowcolor{Gray}
  & & \textbf{{equation}} \\
  \hline
  $R_B, R_C, R_E$ &  change rates & Eqs.~\eqref{G-model} \\
   \hline
  $\mathbf{V}$ & ratio of population & Sec.~\ref{sec:no_equations} \\
  \hline
   $Q$ & food reproduction rate & Eq.~\eqref{Q} \\
   \hline
   $H$ & harvesting rate & Eq.~\eqref{harvesting} \\
   \hline
  $Z$ & normalized food supply & Eq.~\eqref{Z} \\
  \hline
  $F$ & food consumption rate & Eq.~\eqref{F} \\
  \hline
  $G$ & change rate of populations & Eq.~\eqref{G} \\
  \hline
  $\zeta$ & $\zeta$ & Eq.~(\ref{eta}) \\
  \hline
  $L_1$, $L_2$, $L_3$ & $L_1$, $L_2$, $L_3$   & Eqs.~\eqref{L1}---~\eqref{L3} \\
  \hline
\end{tabular}
    \caption{Parameters, variables and equations in  H model (\ref{G-model}), Elite-dominated HANDY model (\ref{4-d-model}), Downward mobility model (\ref{Handymobility}).}
    \label{table-params}   
\end{table}
\begin{table}[H]
\centering
\begin{tabular}{ |l|l|l| }
  \hline
  \rowcolor{Gray}
  \textbf{{ \ \ \ Symbols: }} & {\textbf{{\ \ \ \ \ \ \ \ Symbols: }}}  &  \\
  \rowcolor{Gray}
  \textbf{{Elite-dominated}} &  \ \ \ \ \ \textbf{{HANDY Model}} & \textbf{{\ \ \ Parameter name}} \\
  \rowcolor{Gray}
  \textbf{HANDY Model} & {\textbf{\cite{motesharrei2014human}}} & \\
  \hline
  $C$ & $x_C$ & Commoner population\\
  \hline
  $E$ & $x_E$ & Elite population\\
  \hline
  $B_{Env}$ & $y$  & Environment food \\
  \hline
  $B_{Stor}$ & $w$   & Stored food \\
  \hline
  $Z$ & $w/w_{th}$ & food supply\\
  \hline
   $G(Z)$ & $\beta_C -\alpha_C(w/w_{th})$ & Commoners' change rate \\
  \hline
   $G(\kappa Z)$ & $\beta_E -\alpha_E(w/w_{th})$  & Elites' change rate\\
  \hline
  $\gamma$ & $\gamma \cdot \lambda$  & regeneration factor of $B_{Env}$\\
  \hline
    $\xi_1$ &  $\beta - \alpha_M$ $(\beta = \beta_C=\beta_E)$  & minimum per capita \\
    & & change rate \\
  \hline
   $\xi_2$ & $\beta - \alpha_m$ $(\beta = \beta_C=\beta_E)$  & maximum per capita \\
  & & change rate \\
  \hline
  $\nu$ &  $\delta$  & carrying capacity of food \\
  \hline
  $\sigma$ & $s$ & subsistence food per capita \\
  \hline
\end{tabular}
    \caption{Dictionary for translating this paper's symbology to   \cite{motesharrei2014human}.}
    \label{tab:Symbols}   
\end{table}
\bibliographystyle{unsrtnat}
\bibliography{REF}

\begin{thebibliography}{16}
\providecommand{\natexlab}[1]{#1}
\providecommand{\url}[1]{\texttt{#1}}
\expandafter\ifx\csname urlstyle\endcsname\relax
  \providecommand{\doi}[1]{doi: #1}\else
  \providecommand{\doi}{doi: \begingroup \urlstyle{rm}\Url}\fi

\bibitem[Turchin and Nefedov(2009)]{turchin2009secular}
Peter Turchin and Sergey~A Nefedov.
\newblock \emph{Secular cycles}.
\newblock Princeton University Press, 2009.

\bibitem[Shennan et~al.(2013)Shennan, Downey, Timpson, Edinborough, Colledge,
  Kerig, Manning, and Thomas]{shennan2013regional}
Stephen Shennan, Sean~S Downey, Adrian Timpson, Kevan Edinborough, Sue
  Colledge, Tim Kerig, Katie Manning, and Mark~G Thomas.
\newblock Regional population collapse followed initial agriculture booms in
  mid-holocene europe.
\newblock \emph{Nature communications}, 4:\penalty0 2486, 2013.

\bibitem[Goldberg et~al.(2016)Goldberg, Mychajliw, and Hadly]{goldberg2016post}
Amy Goldberg, Alexis~M Mychajliw, and Elizabeth~A Hadly.
\newblock Post-invasion demography of prehistoric humans in south america.
\newblock \emph{Nature}, 532\penalty0 (7598):\penalty0 232, 2016.

\bibitem[Motesharrei et~al.(2016)Motesharrei, Rivas, Kalnay, Asrar, Busalacchi,
  Cahalan, Cane, Colwell, Feng, Franklin, et~al.]{motesharrei2016modeling}
Safa Motesharrei, Jorge Rivas, Eugenia Kalnay, Ghassem~R Asrar, Antonio~J
  Busalacchi, Robert~F Cahalan, Mark~A Cane, Rita~R Colwell, Kuishuang Feng,
  Rachel~S Franklin, et~al.
\newblock Modeling sustainability: population, inequality, consumption, and
  bidirectional coupling of the earth and human systems.
\newblock \emph{National Science Review}, 3\penalty0 (4):\penalty0 470--494,
  2016.

\bibitem[Turchin(2018)]{turchin2018historical}
Peter Turchin.
\newblock \emph{Historical dynamics: Why states rise and fall}, volume~26.
\newblock Princeton University Press, 2018.

\bibitem[Diamond(2005)]{diamond2005collapse}
Jared Diamond.
\newblock \emph{Collapse: How societies choose to fail or succeed}.
\newblock Penguin, 2005.

\bibitem[Chu and Lee(1994)]{chu1994famine}
CY~Cyrus Chu and Ronald~D Lee.
\newblock Famine, revolt, and the dynastic cycle.
\newblock \emph{Journal of Population Economics}, 7\penalty0 (4):\penalty0
  351--378, 1994.

\bibitem[Stark(2006)]{stark2006funan}
Miriam~T Stark.
\newblock From {F}unan to {A}ngkor: Collapse and regeneration in ancient
  cambodia.
\newblock \emph{After collapse: the regeneration of complex societies}, pages
  144--167, 2006.

\bibitem[Erickson and Gowdy(2000)]{erickson2000resource}
Jon~D Erickson and John~M Gowdy.
\newblock Resource use, institutions, and sustainability: a tale of two pacific
  island cultures.
\newblock \emph{Land Economics}, pages 345--354, 2000.

\bibitem[Motesharrei et~al.(2014)Motesharrei, Rivas, and
  Kalnay]{motesharrei2014human}
Safa Motesharrei, Jorge Rivas, and Eugenia Kalnay.
\newblock Human and nature dynamics (handy): Modeling inequality and use of
  resources in the collapse or sustainability of societies.
\newblock \emph{Ecological Economics}, 101:\penalty0 90--102, 2014.

\bibitem[Lotka(1925)]{alfred1925lotka}
Lotka.
\newblock Lotka.
\newblock \emph{Elements of Physical Biology. Williams and Wilkins}, 1925.

\bibitem[Volterra(1927)]{volterra1927variazioni}
Vito Volterra.
\newblock \emph{Variazioni e fluttuazioni del numero d'individui in specie
  animali conviventi}.
\newblock C. Ferrari, 1927.

\bibitem[Smith(1992)]{smith1992economic}
Vernon~L Smith.
\newblock Economic principles in the emergence of humankind.
\newblock \emph{Economic Inquiry}, 30\penalty0 (1):\penalty0 1, 1992.

\bibitem[Brander and Taylor(1998)]{brander1998simple}
James~A Brander and M~Scott Taylor.
\newblock The simple economics of easter island: A ricardo-malthus model of
  renewable resource use.
\newblock \emph{American economic review}, pages 119--138, 1998.

\bibitem[Meiss(2007)]{meiss2007differential}
James~D Meiss.
\newblock \emph{Differential dynamical systems}, volume~14.
\newblock Siam, 2007.

\bibitem[Haddad and Chellaboina(2011)]{haddad2011nonlinear}
Wassim~M Haddad and VijaySekhar Chellaboina.
\newblock \emph{Nonlinear dynamical systems and control: a Lyapunov-based
  approach}.
\newblock Princeton University Press, 2011.

\end{thebibliography}
\end{document}